\documentstyle[amssymb,amsfonts]{amsart}

\newenvironment{proof}{\noindent {\bf Proof} }{\endprf\par}
\def \endprf{\hfill  {\vrule height6pt width6pt depth0pt}\medskip}

\theoremstyle{plain}
  \newtheorem{theorem}[subsection]{Theorem}

  \newtheorem{lemma}[subsection]{Lemma}

\theoremstyle{remark}

\theoremstyle{definition}

\include{psfig}

\begin{document}

\title[On Freiman's theorem in Nilpotent groups]{On Freiman's theorem in Nilpotent groups}

\author{David Fisher}
\address{Department of Mathematics, Indiana University, Bloomington}
\email{ fisherdm@@indiana.edu}

\author{Nets Hawk Katz}
\address{Department of Mathematics, Indiana University, Bloomington}
\email{nhkatz@@indiana.edu}

\author{Irine Peng}
\address{Department of Mathematics, Indiana University, Bloomington}
\email{kanamejun@@gmail.com}

\vspace{-0.3in}

\begin{abstract} We generalize a result of Tao which extends Freiman's theorem to the
Heisenberg group. We extend it to simply connected nilpotent Lie groups of arbitrary
step.
\end{abstract}

\maketitle

\section{Introduction}

Given a set $A$ in an abelian group, we say it has small additive doubling if
$$|A+A| \leq K|A|,$$
where $K$ is a constant which is small in an appropriate sense.

The aim of the classical Freiman's theorem is to show that sets
having small additive doubling exhibit a lot of structure.

Terry Tao in \cite{Tao} studied the analogous situation for sets of small multiplicative
tripling in a Heisenberg group. He showed that such a set can be mapped into an abelian
group in such a way that it has additive structure which is consistent with commutation.
Our aim is to obtain such a result in the setting of general simply connected
nilpotent Lie groups. Our only tool will be a rather direct application of the
Baker-Campbell-Hausdorff formula.

We will investigate the structure of a subset $A \subset N$ which is an
{\it approximate multiplicative group}.
We recall an approximate multiplicative group is a set $A$ with the property
$$AA \subset \bigcup_{l=1}^k  x_l A,$$
where the $x_l$'s are elements of $N$ and we refer to $k$ as the multiplicative constant
of $A$. We will restrict to approximate multiplicative groups which are symmetric:
$A=A^{-1}$. The relations
between symmetric approximate multiplicative groups and sets with small tripling are
discussed in \cite{Tao} and \cite{TaoVu}.

Our goal is to prove the following theorem:

\begin{theorem} \label{TheBigOne} Let $A$ be a symmetric approximate multiplicative group
in a simply connected, nilpotent,
$n$-step Lie group $N$.
\begin{enumerate}
\item \label{additive} We have that for any small set of small integers $k_1, k_2, \dots k_l$
that
$$|\log(A^{k_1}) + \log (A^{k_2})+ \dots \log(A^{k_l})| \lesssim |\log (A)|,$$
with the constants depending only on the multiplicative constant of $A$. the step $n$, the maximum of the
$|k_j|$'s and $l$.

\item \label{multiplicative}  Let $B_0=\log(A)$ and $B_j$ be $[B_0,B_{j-1}]$. (By nilpotence this
sequence terminates.) Then for each $j$, there is a number $l$, integers $k_1, \dots, k_l$
and rationals $q_1,\dots q_l$ depending just on $j$ and $n$ so that
$$B_j \subset q_1 \log(A^{k_1}) + q_2 \log(A^{k_2}) + \dots q_l \log(A^{k_l}).$$

\end{enumerate}
\end{theorem}

Morally, this result says that if $A$ is a symmetric approximate
multiplicative group then $\log(A)$ is very close to being a Lie
algebra. In particular, the first says that the $\log$ of small powers of $A$ all have small additive
doubling and are almost additively parallel. The second part says that the iterated commutators
of $\log(A)$ (which by nilpotence live in progessively smaller subspaces of the Lie algebra)
are also additively compatible with sets from the first part. In particular, a set $A$ having just
these two properties automatically has small mutliplicative tripling by applying the
Baker-Campbell-Hausdorff formula so that we essentially have a characterization of such sets.

 Further, one can apply classical Freiman's theorem to
$\log(A)$ and then use the second condition in our theorem to impose
further restrictions on the resulting generalized arithmetic
progression. It is also fairly easy to adapt our arguments to the
setting of torsion nilpotent groups, at least when all elements have
order that is large enough.  Here large enough is in terms of
various constants arising in the Baker-Campbell-Hausdorff formula.
Combining this with the theorem above and known facts about the
structure of nilpotent groups gives a result for all nilpotent
groups without small torsion.

Our argument uses nilpotency in an essential way. (The
Baker-Campbell-Hausdorff formula is finite!) There is some
speculation among experts that all symmetric multiplicative
subgroups of an arbitrary group $G$ may arise from subsets of
nilpotent groups like those arising in our theorem see e.g. \cite[Chapter 3.2]{Taoblog} and references
there.  The recent preprint by Helfgott can be seen as further evidence for this speculation
\cite{Helfgott}.  As mentioned there,   this problem for general $G$  seems potentially related to Gromov's theorem on groups of polynomial growth and Kleiner's recent effective proof of this
theorem \cite{Gromov, Kleiner}.  The case of linear groups already
seems difficult and interesting and should be independent of analogues of
Gromov's theorem.  Particularly in the linear case, it seems that
ideas from work on uniform exponential growth and uniform independence 
might be relevant \cite{EskinMozesOh,
BreuillardGelander, Breuillard}.


{\bf Acknowledgements:} The first author is partially supported
by NSF grant 0643546 and a fellowhsip from the Radcliffe Institute for
Advanced Studies. The second author is partially supported by NSF grant 0653763.

\section{Proofs}

Let $N$ be a simply connected nilpotent Lie group of nilpotency $n$.
  It is known that $\exp$, the exponential map, is a diffeomorphism and
we let $\log: N \rightarrow \frak{n}$
be its inverse.
For any $x, y \in N$, $q \in \mathbb{Q}$, we write
$x^{q}$ to mean $\exp( q \log(x))$, and $c(x,y)$ to mean their commutator $xyx^{-1}y^{-1}$.
For any positive integers $L$, and $j \geq 1$, we
write $\Omega_{j}^{L}$ to mean all the possible maps from $\{1,2,3, \cdots j \}$ to
$\{ 1,2, 3, \cdots L \}$.
If $\alpha \in \Omega_{j}^{L}$, we define $|\alpha|:=j$.  Whenever we write $\prod_{\iota
\in \Omega_{j}^{L}} x_{\iota}$, we mean multiplying from left to right with respect to a
pre-fixed linear order on $\Omega_{j}^{L}$. \\

Recall the Campbell-Baker-Hausdorff formula, which says that for any $x, y \in N$,
\begin{equation} \label{CBH}
\log(x) + \log(y) = \log(xy) +
\sum_{j=2}^{n} \sum_{\alpha \in \Omega_{j}^{2}} t_{\alpha} h_{\alpha}  \end{equation}

\noindent where each $t_{\alpha}$ is a rational number and $h_{\alpha}=
(ad X_{1} ) \circ (ad X_{2} ) \cdots \circ (ad X_{|\alpha|-1}) X_{|\alpha|}$, with
$X_{j}$ equal to
$\log(x)$ if $\alpha(j)=1$ and equal to $\log(y)$ if $\alpha(j)=2$.

\begin{lemma} \label{useful identity}
Take $x_{1}, \cdots x_{L} \in N$, and let $X_{i}=\log(x_{i})$.  For any $\alpha
\in  \Omega_{j}^{L}$, we write $h_{\alpha}$ to mean
$(ad X_{\alpha(1)})  \circ \cdots (ad X_{\alpha(j-1)}) X_{\alpha(j)}$, $H_{\alpha}$ to
mean
$c(x_{\alpha(1)},c(x_{\alpha(2)}, \cdots c(x_{\alpha(j-1)},x_{\alpha(j)})))$.  Then
\begin{enumerate}
\item \label{iteratedCBH} \[ \log(x_{1}.x_{2} \cdots x_{L}) = \sum_{i=1}^{L} X_{i} +
\sum_{j=2}^{n} \sum_{\alpha \in \Omega_{j}^{L}} c_{\alpha} h_{\alpha}, \mbox{    }
c_{\alpha} \in \mathbb{Q} \]

\item \label{commutators} For any $j$ and $\alpha \in \Omega_{j}^{L}$, (remember $L$ is fixed)
\[ \log(H_{\alpha}) = h_{\alpha} + \sum_{\ell=j+1}^{n} \sum_{\beta \in \Omega_{\ell}^{L}}
s_{\beta} h_{\beta}, \mbox{      } s_{\beta} \in \mathbb{Q} \]

\item  \label{logproduct} For any $1 \leq J \leq n$,
There are rational numbers $\beta_{2},\beta_{3}, \cdots \beta_{n}$
depending only on $J$ and $L$ so that,
\[ \sum_{i=1}^{L} X_{i} = \log(x_{1}x_{2}\cdots x_{L}) + \sum_{i=2}^{J} \beta_{i}
\log(\tilde{M}_{i}) + \sum_{i=J+1}^{n} \sum_{\alpha \in \Omega_{i}^{L}} c''_{\alpha}
h_{\alpha} \]
\noindent where $\tilde{M}_{i}=\prod_{\alpha \in \Omega_{i}^{L}}
\left( H_{\alpha} \right)^{m_{\alpha}}$ for
integers $m_{\alpha}$.

\item \label{bigproduct} For any $1 \leq j \leq n$, There are integers
$c_{1},c_{2} \cdots c_{n}$ (depending only on $j$ and $L$) such that whenever
an integer $T$ is divisible by each $c_{j}$
\[ \sum_{\iota} T X_{\iota} = \log( \prod_{i=1}^{L} x_{i}^{T}  \prod_{\ell=2}^{j}
\prod_{\alpha \in
\Omega_{\ell}^{L}} \tilde{H}_{\alpha}) + \sum_{\ell=j+1}^{n} \sum_{\alpha
\in \Omega_{\ell}^{L}} s'_{\alpha} h_{\alpha},
 \mbox{            } s'_{\alpha} \in \mathbb{Q} \]

\noindent where each $\tilde{H}_{\alpha}=c(x^{m_{\alpha(1)}}_{\alpha(1)},
c(x^{m_{\alpha(2)}}_{\alpha(2)}, \cdots ,
c(x^{m_{\alpha(j-1)}}_{\alpha(j-1)},x^{m_{\alpha(j)}}_{\alpha(j)})))$ for some
integers $m^{\alpha(i)}$ (depending on $T$). \end{enumerate} \end{lemma}

In particular, in the case of (\ref{bigproduct}), we are interested in the case $j=n$.
We have stated the result with the parameter $j$ so that it is easy to write down a
proof which is an induction.

\begin{proof}
The proof of (\ref{iteratedCBH}) is a repeated application of Campbell-Baker-Hausdorff.
Suppose the claim
is true up to $J$, then
\[ \log(x_{1}x_{2} \cdots x_{J}) = \sum_{i=1}^{J} X_{i} + \sum_{\ell=2}^{n}
\sum_{\alpha \in
\Omega_{\ell}^{J}} c_{\alpha} h_{\alpha} \]

Applying (\ref{CBH}) to $x_{1}x_{2}\cdots x_{J}$ and $x_{J+1}$ gives us
\[ \log(x_{1}x_{2} \cdots x_{J}.x_{J+1}) = \log(x_{1}x_{2}\cdots x_{J}) + \log(x_{J+1}) +
\sum_{j=2}^{n} \sum_{\alpha \in \Omega_{j}^{2}} t_{\alpha} h_{\alpha}  \]

\noindent where each $h_{\alpha}=(ad X_{1} ) \circ (ad X_{2} ) \cdots \circ (ad X_{|\alpha|-1}) X_{|\alpha|}$ with $X_{j}$ equal either
$\log(x_{J+1})$ or $\log(x_{1}x_{2}\cdots x_{J})$.  Substituting the latter by the
inductive hypothesis yields \ref{iteratedCBH} for $J+1$. \\

We now prove (\ref{commutators}) by induction on $|\alpha|$.
The base step (when $|\alpha|=2$) is obtained by combining (\ref{CBH}) with the following
\[ \log(e^{a}e^{b}e^{-a})= b+ \sum_{j=1}^{n-1} \frac{1}{j!} (ad(a))^{j}(b) \]

Suppose now that (\ref{commutators}) is true up to $J$ many arguments.  Then,
\begin{eqnarray*}
\log(c(x_{1}, \underbrace{ c(x_{2}, \cdots ,c(x_{J-1},c(x_{J},x_{J+1}))) } ))
&=& [X_{1}, \log(c(x_{2}, \cdots ,c(x_{J-1},c(x_{J},x_{J+1}))))] \\
&& + \sum_{\ell=2}^{n} \sum_{\alpha \in \Omega_{\ell}^{2}} t_{\alpha} h_{\alpha} \\
&=& [X_{1},[X_{2}, [\cdots, [X_{J},X_{J+1}]]]] + \sum_{\ell=J+2}^{n} \sum_{\alpha \in \Omega_{\ell}^{J+1}} s_{\alpha} h_{\alpha}  \end{eqnarray*}

\noindent where the first line comes from applying the base step to $x_{1}$ and
the underlined expression, and the second line comes from applying the inductive assumption.\\

Equation (\ref{logproduct}) can be proved similarly by inducting on $J$.  The case $J=1$
is just (\ref{iteratedCBH}).

Suppose it is true
for $J$.  Then
\[ \sum_{i=1}^{L} X_{i} = \log(x_{1}x_{2}\cdots x_{L}) + \sum_{i=2}^{J} \beta_{i}
\log(\tilde{M}_{i}) + \sum_{\alpha \in \Omega_{J+1}^{L}} t_{\alpha} h_{\alpha} +
\sum_{i=J+2}^{n} \sum_{\alpha \in \Omega_{i}^{L}} c''_{\alpha} h_{\alpha} \]

Let $\beta_{J+1}$ be 1 over the smallest common multiple of the denominators of
$t_{\alpha}$'s, where $\alpha$ ranges over $\Omega_{J+1}^{L}$.  Then the second summand above can be replaced by
\begin{eqnarray*}
\sum_{\alpha \in \Omega_{J+1}^{L}} t_{\alpha} h_{\alpha} &=& \beta_{J+1} \sum_{\alpha
\in \Omega_{J+1}^{L}} m_{\alpha}h_{\alpha}  \mbox{,            } m_{\alpha} \in \mathbb{Z} \\
&=& \beta_{J+1} \sum_{\alpha \in \Omega_{J+1}^{L}} \log(H_{\alpha}^{m_{\alpha}} )
+ \sum_{i=J+2}^{n} \sum_{\alpha \in \Omega_{i}^{L}} t'_{\alpha}h_{\alpha}  \\
&=& \beta_{J+1} \log ( \prod_{\alpha \in \Omega_{J+1}^{L}} H_{\alpha}^{m_{\alpha}} )
+ \sum_{j=J+2}^{n} \sum_{\alpha \in \Omega_{j}^{L}} t''_{\alpha} h_{\alpha} \end{eqnarray*}

\noindent where the second line comes from applying (\ref{commutators}) to each term
in the sum on the
right hand side of first line, and the third line comes from applying (\ref{iteratedCBH})
 followed by
(\ref{commutators}) to the first sum on
the second line. \\

We now prove (\ref{bigproduct}) by induction on $j$.  The base step is just (\ref{CBH}).
Suppose it is true
up to $J$.  If we let $c_{J+1}$ to be the lowest common multiple between $c_{J}$
and denominators of $s'_{\alpha}$'s with $|\alpha|=J+1$, then using $(3)$ of the lemma, we have:
\begin{equation}\label{inductive step}
\sum_{i=1}^{L} c_{J+1}X_{i} = \log(\prod_{i=1}^{L} x_{i}^{c_{J+1}} \prod_{j=2}^{J}
\prod_{\beta \in \Omega_{j}^{L}} \tilde{H}_{\beta}) + \sum_{\alpha \in \Omega_{J+1}^{L}} \tilde{h}_{\alpha}
+  \sum_{\ell=J+2}^{n} \sum_{\beta \in \Omega_{\ell}^{L}} s'_{\beta} h_{\beta} \end{equation}
\noindent where $\tilde{h}_{\alpha}=ad (m_{\alpha(1)}X_{\alpha(1)} ) \circ ad (m_{\alpha(1)}X_{\alpha(1)} )
\cdots ad (m_{\alpha(J)}X_{\alpha(J)}) (m_{\alpha(J+1)}X_{\alpha(J+1)}) $ for integers
$m_{\alpha}$'s.

Applying (\ref{commutators}) to the second term on the right hand side, we have
\[ \sum_{i=1}^{L} c_{J+1}X_{i} = \log(\prod_{i=1}^{L} x_{i}^{c_{J+1}} \prod_{j=2}^{J}
\prod_{\beta \in \Omega_{j}^{L}} \tilde{H}_{\beta}) + \sum_{\alpha \in \Omega_{J+1}^{L}} \log(\tilde{H}_{\alpha})
+  \sum_{\ell=J+2}^{n} \sum_{\beta \in \Omega_{\ell}^{L}} s''_{\beta} h_{\beta} \]

\noindent On the right hand side, we can combine the first term and each summand in the second term one at a time by applying (\ref{CBH})
and using (\ref{iteratedCBH}), (\ref{commutators}) to get the desired equality.
\end{proof}

We proceed towards the proof of the main theorem but we first need the
following Lemma.

\begin{lemma}
\label{lemma:easy}
For all $n$ there exists an integer $m$ depending only on $n$ so that if
$a,b \in G$ with $G$ a simply
connected $n$-step nilpotent Lie group then
$$m (\log a + \log b) = \log (w_{m,n}(a,b)),$$
with $w_{m,n}(a,b)$ a word in $a,b,a^{-1},b^{-1}$ with length $l(m,n)$ depending
only on $m,n$.
\end{lemma}

\begin{proof} This is an immediate consequence of Lemma \ref{useful identity} part
(\ref{bigproduct}).
\end{proof}

Now we are ready to prove the main theorem.

\begin{proof}

Our first goal is to show that
$$|\log(A^{k_1}) + \log (A^{k_2})+ \dots \log(A^{k_l})| \lesssim |\log (A)|,$$
where the constants depend only on the step $n$, the maximum of the
$|k_j|$'s, $l$  and the multiplicative constant of $A$.
This follows directly from Lemma \ref{lemma:easy}. Applying that lemma repeatedly, we find
that every element of $\log(A^{k_1}) + \log (A^{k_2})+ \dots \log(A^{k_l})$ is contained in
${1 \over m} \log(A^l(m,n))$.
Because $A$ is an approximate multiplicative subgroup we have that
$$|A^l(m,n)| \lesssim |A|,$$
and we have proved the first part.

Next we must show that
$$B_j \subset q_1 \log(A^{k_1}) + q_2 \log(A^{k_2}) + \dots q_l \log(A^{k_l}).$$

We prove this by induction. Clearly it is true for $B_0$. Let us suppose that it is true for
$B_{j-1}$. Then applying lemma \ref{lemma:easy} to the induction hypothesis repeatedly,
 we find $m$ and $k$ depending only on $j$
and $n$ so that for any $b \in B_{j-1}$ we have
$$e^{mb} \in A^k.$$

We recall again the identity
$$\log(e^a e^{mb} e^{-a}) = mb + \sum_{j=1}^{n-1} {1 \over j!} ad(a)^j mb.$$
For any $1 \leq s \leq n-1$ with $s$ an integer, this implies that
$$\log(e^{sa} e^{mb} e^{-sa})=mb + \sum_{j=1}^{n-1} {s^j \over j!} ad(a)^j mb.$$
Viewing this as a system of $n$ linear equations for the unknowns $ad(a)^j b$, we can solve
by inverting the Vandermonde matrix, and we find
finding rationals $q_1 ,\dots q_n$ depending only on $n$ and $j$ so that if $a \in \log(A),$
and $b \in B_{j-1}$ then
we have
$$[a,b] \in q_1 \log (A^{k+2}) + q_2 \log (A^{k+4}) + \dots + q_n \log (A^{k+2n}).$$
Thus we are done.

\end{proof}

\end{document}